\documentclass{amsart} 

\numberwithin{equation}{section}

\theoremstyle{plain}
\newtheorem{theorem}{Theorem}[section]
\newtheorem{proposition}[theorem]{Proposition}         
\newtheorem{corollary}[theorem]{Corollary} 
 
\newtheorem{definition}[theorem]{Definition}  
\newtheorem{notation}[theorem]{Notation}   

\theoremstyle{definition}  
\newtheorem{example}[theorem]{Example} 

\newcommand{\no}{\noindent} 
\newcommand{\bp}{\begin{pmatrix}} 
\newcommand{\ep}{\end{pmatrix}}

\newcommand{\C}{\mathbb C}   
\newcommand{\R}{\mathbb R}
\newcommand{\Z}{\mathbb Z}

\newcommand{\FF}{\mathbb F}
\renewcommand{\P}{\mathbb P}

\newcommand{\al}{\alpha}
\newcommand{\be}{\beta} 
\newcommand{\ga}{\gamma}
\newcommand{\de}{\delta}
\newcommand{\la}{\lambda}
\newcommand{\si}{\sigma} 
 
\newcommand{\La}{\Lambda}

\newcommand{\Om}{\Omega}

\renewcommand{\th}{\theta}
\newcommand{\om}{\omega}

\newcommand{\Ga}{\Gamma}
\newcommand{\na}{\nabla}

\newcommand{\calM}{\mathcal{M}}
\newcommand{\calN}{\mathcal{N}}
\newcommand{\calH}{\mathcal{H}}
\newcommand{\calF}{\mathcal{F}}
\newcommand{\calO}{\mathcal{O}}

\newcommand{\calT}{\mathcal{T}}
\newcommand{\calR}{\mathcal{R}}

\newcommand{\rank}{\mathrm{rank}}

\newcommand{\Ker}{\mathrm{Ker}}

\newcommand{\Hom}{\mathrm{Hom}}

\newcommand{\Map}{\mathrm{Map}}
\newcommand{\PD}{\mathrm{PD}}
\newcommand{\Span}{\mathrm{Span}}

\newcommand{\cyclic}{\mathrm{cyclic}}

\newcommand{\orbi}{\mathrm{orbi}}

\newcommand{\sub}{\subseteq}

\newcommand{\pr}{\prime}

\newcommand{\st}{\ \vert\ }
\renewcommand{\ll}{\lq\lq}
\newcommand{\rr}{\rq\rq\ }
\newcommand{\rrr}{\rq\rq}
\newcommand{\lan}{\langle}

\newcommand{\ran}{\rangle}

\renewcommand{\b}{\partial}
\newcommand{\bx}{\partial_{x}}
\newcommand{\by}{\partial_{y}}
\newcommand{\bt}{\partial_{t}}
\newcommand{\db}{\bar{\partial}}
\newcommand{\ABC}{\lan A \vert B \vert C \ran}
\newcommand{\subD}{_{\scriptsize D}}

\newcommand{\HolD}{\operatorname {Hol}\subD}
\newcommand{\mon}{ {\operatorname {m}} }

\newcommand{\GL}{\textrm{GL}}

\renewcommand{\O}{\textrm{O}}
\newcommand{\SU}{\textrm{SU}}
\newcommand{\U}{\textrm{U}}
\renewcommand{\S}{\textrm{S}}
\newcommand{\Sp}{\textrm{Sp}}
\newcommand{\glnc}{\GL_n\C}
\newcommand{\glsc}{\GL_{s+1}\C}
\newcommand{\glsr}{\GL_{s+1}\R}

\newcommand{\GC}{G^{\C}}

\newcommand{\frakglsc}{\frak g\frak l_{s+1}\C}

\newcommand{\g}{{\frak g}}
\newcommand{\gc}{{\frak g^{\C}}}
\renewcommand{\k}{{\frak k}}

\newcommand{\m}{{\frak m}}
\newcommand{\mc}{{\frak m^{\C}}}

\newcommand{\Gr}{\textrm{Gr}}

\newcommand{\smallsum}{  {\textstyle \text{$\sum$} }}
\newcommand{\F}{\calF}
\newcommand{\h}{ {\hbar} }
\newcommand{\one}{{\mathbf 1}_{\frac12}}
\renewcommand{\i}{i}

\begin{document}

\title[Quantum cohomology]{Differential equations aspects of quantum cohomology}  

\author{Martin A. Guest}      

\date{}   

\maketitle 

The concept of quantum cohomology arose in string theory around 20 years ago. Its mathematical foundations were established around 10 years ago,  based on the theory of Gromov-Witten invariants.  There are two approaches to Gromov-Witten invariants, via symplectic geometry or via algebraic geometry.  Both approaches give the same results for the three-point Gromov-Witten invariants of familiar manifolds $M$ like Grassmannians and flag manifolds, and these invariants may be viewed as the structure constants of the quantum cohomology algebra $QH^\ast M$,  a modification of the ordinary cohomology algebra $H^\ast M$.

However, the name \ll quantum cohomology\rr may be misleading.  On the one hand, the \ll quantum\rr and \ll cohomology\rr aspects are somewhat removed from the standard ideas of quantum physics and cohomology theory.  On the other hand, there are strong relations between quantum cohomology and several other areas of mathematics:  symplectic geometry and algebraic geometry, of course, but also differential geometry, the theory of integrable systems (soliton equations), and even number theory.

In these lectures\footnote{This article is based on
lectures given at the summer school \ll Geometric and Topological Methods for Quantum Field Theory\rrr, 
Villa de Leyva, 2007.}
we shall focus on the {\em quantum differential equations} as the fundamental concept (due to Alexander Givental:
\cite{Gi95-1}, \cite{Gi95-2}, \cite{Gi96}) which encapsulates many aspects of quantum cohomology.  This is \ll more elementary\rr than the definition of Gromov-Witten invariants, in the same way that de Rham cohomology is \ll more elementary\rr than the definition of simplicial or singular cohomology.  In addition, it is essential for understanding the relation between quantum cohomology and the theory of integrable systems, which is becoming increasingly important. 
The language of D-module theory is very convenient for this purpose.  It provides a unified way to think about (a) classical integrable systems such as the KdV equation, (b) integrable systems in differential geometry, such as harmonic maps, and (c) quantum cohomology.  

The abstract theory of D-modules is well-developed, but not widely used by nonspecialists. One goal of these lectures is to give some motivation for D-module theory (sections \ref{pde},\ref{extensions}), and to advertise some of its uses (sections \ref{qde},\ref{applications}).  A second goal is to explain in simple terms how quantum cohomology is related to other parts of mathematics.  These links have deep origins and are still evolving.  It can be difficult to grasp them from articles which use haphazardly very technical language from symplectic geometry, algebraic geometry, and singularity theory, especially when some of the links are conjectural.

The lecture series  \ll From quantum cohomology to integrable systems\rr  (based on the book \cite{Gu08}) traced a path from ordinary cohomology theory to the quantum differential equations and their role in the theory of integrable systems. The introductory lectures on cohomology and quantum cohomology do not appear here (they are the subject of an earlier survey article \cite{Gu06}).
The differential equations aspects of the lectures have been expanded slightly, and the applications have been gathered together in section \ref{applications}. The lectures contained various concrete examples from \cite{Gu08}, which have been omitted here to save space.
 
I am very grateful to the organisers of the summer school for their invitation to give these lectures and for their careful planning which resulted in an effective and pleasant environment.  I thank Ramiro Carrillo-Catal\'an for preparing a Spanish language version of the original lecture notes.
I am grateful to PIMS at the University of British Columbia for its hospitality in August 2007, where some of this material was written.  I thank the referee for refreshingly frank comments and for urging me to write a much better article, which I hope to do some day.
This research was supported by a grant from the JSPS.

\section{Linear differential equations and D-modules}\label{pde}

Consider the linear ordinary differential equation
\[
(\b^{s+1} + a_s\b^s + \dots + a_1\b + a_0)y=0
\]
for $y=y(z)$, where
$
\b = \b_z=\frac{d}{d z}
$
and the coefficients $a_0,\dots,a_s$ are functions of the
complex variable $z$. 
We assume that $a_0,\dots,a_s$ belong to the ring
$\calH=\calH_z$ of holomorphic functions on an open disk  $N=N_z$  in $\C$.   We shall write $T=\b^{s+1} + a_s\b^s + \dots + a_1\b + a_0$.

Almost all textbooks on differential equations contain a proof of the following basic result:

\begin{theorem}  
For any point $z_0\in N$ and any  values $c_0,\dots,c_s\in\C$, there is a unique solution $y\in\calH$ of $Ty=0$ which satisfies the initial conditions $y(z_0)=c_0,y^\pr(z_0)=c_1,\dots, y^{(s)}(z_0)=c_s$. 
\end{theorem}  

\begin{corollary}  
The set of all holomorphic solutions  (on $N$)
of the  o.d.e.\  $Ty=0$
is a vector space of dimension $s+1$.
\end{corollary} 

Introducing new variables
\[
y_0=y,\ \ y_1=\b y=y^\pr,\ \ \dots,\ \ y_s={\b}^s y=y^{(s)}
\]
we may convert the above scalar equation of order $s+1$
to an equivalent system of $s+1$  first order equations of the form $\b Y=AY$:
\[
\b
\bp
y_0 \\ \vdots \\ y_{s-1} \\ y_s
\ep
=
\bp
0 & 1 &  & \\
 & \ddots & \ddots & \\
  & & 0 & 1 \\
  -a_0 & \cdots &  -a_{s-1} & -a_s
  \ep
\bp
y_0 \\ \vdots \\ y_{s-1} \\ y_s
\ep.
\]
\begin{corollary}  
The set of all holomorphic solutions  (on $N$)
of the  system  $\b Y = A Y$
is a vector space of dimension $s+1$.
\end{corollary} 

Let us choose a basis $y_{(0)},\dots,y_{(s)}$ of
solutions of the scalar equation.  The corresponding vector functions 
\[
Y_{(i)} = 
\bp
y_{(i)} \\ \b y_{(i)}  \\  \vdots \\ {\b}^s y_{(i)} 
\ep,
\quad\quad
0\le i\le s
\]
constitute a basis of solutions of $\b Y=AY$.
The fundamental solution matrix
\[
H=
\bp
\vert & & \vert \\
Y_{(0)} & \cdots & Y_{(s)} \\
\vert & & \vert
\ep
\]
satisfies $\b H = A H$, 
and it takes values in the group $\glsc$ of invertible $(s+1)\times(s+1)$
complex matrices. This \ll matrix o.d.e.\rr is a third incarnation of the original equation.  Its solutions correspond to intial conditions $H(z_0)=C\in \glsc$.  

The matrix $A$ 
depends on the definition of $y_0,\dots,y_s$.  Instead of using the successive derivatives of $y$, let us set
\[
y_0=P_0y,\quad y_1=P_1 y,\quad \dots,\quad y_s=P_s y
\]
where $P_0,\dots,P_s$ are differential operators.  This leads to an equivalent system of first order equations if the equivalence classes
$[P_0],\dots, [P_s]$ form a basis (over $\calH$)  of the \ll D-module\rr
\[
\calM = D/(T),
\]
where $D$ denotes the ring of differential operators (polynomials in $\b$ with coefficients in $\calH$) and $(T)$ denotes the left ideal generated by $T$. The ring operations on $D$ are addition and composition of differential operators.  

Each such choice of basis corresponds to a way of converting the scalar equation to a matrix equation $\b Y=AY$; the matrix $A$ is given 
explicitly by\footnote{The notation $\b P_j$ here means composition of differential operators;  this conflicts with our earlier usage of $\b f$ to mean $\b f/\b z$.  We shall just rely on the context to distinguish these:  $\b f$ means the function $\b f/\b z$ when used in a differential equation, while in a D-module computation it is the same as the operator  $f\b + \b f/\b z$.}
$\b [P_j]=[\b P_j]= \sum_{k=0}^s A_{jk} [P_k]$.  
It should be noted that $\calM$ is an infinite-dimensional complex vector space,  indeed it can be identified with the space 
$\Map(N,\C^{s+1})$ of (holomorphic) maps from $N$ to $\C^{s+1}$. 

This discussion can be generalized to 
partial differential equations. We shall be concerned only with \ll overdetermined\rr linear systems of p.d.e., which share many common features with linear o.d.e, in particular finite-dimensionality of the solution space.  
Let $N=N_{z_1,\dots,z_r}$ be a fixed open polydisk in $\C^r$. 
Writing $\b_1 = {\b}/{\b z_1}, \dots,  \b_r = {\b}/{\b z_r},$
we consider a system of p.d.e.\ 
\[
T_1y=0,\quad\dots,\quad T_uy=0
\]
for a scalar function $y(z_1,\dots,z_r)$ on $N$. 
The $T_i$ are differential operators, that is, polynomials in $\b_1,\dots,\b_r$ with coefficients in the ring $\calH=\calH_{z_1,\dots,z_r}$ (functions of $z_1,\dots,z_r$ which are holomorphic in $N$).

In contrast to the o.d.e.\ case, it is not at all clear whether the solution space is finite-dimensional (or what the dimension is).  
The concept of D-module and the closely related concept of flat connection are essential at this point. We shall just give a brief discussion, referring to \cite{Ph79} for the general theory.
Let $D=D_{z_1,\dots,z_r}$ be the ring of differential operators generated by $\b_1,\dots,\b_r$ with
coefficients in the ring $\calH$ of holomorphic functions on $N$. 
Let 
\[
\calM=D/(T_1,\dots,T_u),  
\]
where $(T_1,\dots,T_u)$ means the
left ideal generated by the differential operators $T_1,\dots,T_u$.  

\no{\bf Assumption:} 
$\calM$ is a free module of rank $s+1$ over $\calH$.

\no This is a strong assumption, but the following proposition (whose proof consists merely of unravelling the definitions) allows us to conclude that   the solution space of the original system has dimension $s+1$:

\begin{proposition}\label{hom}  The map
\begin{align*}
\th: \{ y\st T_iy=0, 1\le i\le u \}  \ \   &   \to  \ \ 
\Hom_D(\calM,\calH),\\
y\ \ \  &\mapsto \ \  ( [X] \mapsto Xy )
\end{align*}
is an isomorphism of complex vector spaces.
\end{proposition}

As in the o.d.e.\ case, the D-module point of view allows us to make clear the relation between scalar and matrix equations, under the above assumption.  However, there is an important new ingredient, a certain flat connection, which leads to the relation with integrable systems.  We shall therefore review the whole procedure carefully, taking the opportunity to introduce some further notation.

\no{\em How to convert a scalar system to a matrix system.} 

Let $[P_0],\dots,[P_s]$ be a basis of $\calM$ over $\calH$.
We define 
$(s+1)\times(s+1)$ matrix functions  $\Om_1,\dots,\Om_r$ by
\[
\b_i [P_j]=[\b_i P_j] = \sum_{k=0}^s (\Om_{i})_{kj} [P_k].
\]
We set $\Om=\sum_{i=1}^r \Om_i dz_i$.  
Then $\na=d+\Om$ defines a connection in the trivial vector bundle
$N\times \C^{s+1}$, whose space of sections is
$\Map(N,\C^{s+1}$).
Namely,
\begin{align*}
\na_{\b_i}(\smallsum_{j=0}^s f_j [P_j]) &=
\sum_{j=0}^s \b f_j/\b z_i [P_j] + \smallsum_{j=0}^s f_j \na_{\b_i}[P_j] \\
&=
\smallsum_{j=0}^s \b f_j/\b z_i [P_j] + \smallsum_{j,k=0}^s f_j (\Om_{i})_{kj} [P_k].
\end{align*}

Let us now recall a well known fact about such connections (see section 4.5 of \cite{Gu08}).

\begin{theorem}\label{flat} The following statements are equivalent:

\no(1) The connection $d+\Om$ is flat (i.e.\ has zero curvature).

\no(2) $d\Om + \Om\wedge\Om=0$.

\no(3) $[\b_i+\Om_i,\b_j+\Om_j]=0$ for all $i,j$.

\no(4) $\Om=L^{-1}dL$ for some $L:N\to \glsc$ (for this it is essential that $N$ is simply connected).
\end{theorem}

Using this we obtain:

\begin{proposition} 
The connection $\na=d+\Om$ obtained from $\calM$ is flat.
\end{proposition}

\begin{proof} 
Since $\b_i \b_j = \b_j \b_i$, we have
$\b_i(\b_j [P]) = \b_j(\b_i [P])$ for any $[P]\in\calM$, hence 
$(\b_i+\Om_i)(\b_j+\Om_j)f=(\b_j+\Om_j)(\b_i+\Om_i)f$ for any $f$, i.e.\  $[\b_i+\Om_i,\b_j+\Om_j]=0$.
\end{proof}

The map $L^t$ in (4)  can be regarded as a fundamental solution matrix for the
system $(\b_i-\Om^t_i)Y=0$, $1\le i\le r$.  If we introduce
$A_i=\Om_i^t$ we obtain a matrix system
\[
\b_iY=A_iY, \quad 1\le i\le r
\]
of the required form.  It should be noted that $(\b_i-\Om^t_i)Y=0$ is the equation for parallel (covariant constant) sections with respect to the {\em dual} connection $\nabla^\ast=d-\Om^t$,  rather than $\nabla=d+\Om$ itself.  The identification of the solution space of the system with $\Ker\nabla^\ast$ may be regarded as a matrix version of 
Proposition \ref{hom}.  

To summarize, we can say that the choice of basis $[P_0],\dots,[P_s]$ produces a matrix system from a scalar system in the following way:

\begin{proposition} The map
$
y\longmapsto
Y=(P_0y,\dots,P_sy)^t
$
from the solution space
\[
\{ y\st T_jy=0, 1\le j\le u  \}
\ \cong\  \Hom_D(\calM,\calH)
\]
to the solution space
\[
\{ Y\st \b_iY = A_i Y, 1\le i\le r \}
\ = \ \Ker \na^\ast
\]
is an isomorphism of $(s+1)$-dimensional vector spaces.
\end{proposition}

The appearance of the dual connection $\nabla^\ast$ and the dual D-module $\calM^\ast=\Hom_\calH(\calM,\calH)$ (of which $\Hom_D(\calM,\calH)$ is a subspace) is an important feature of the construction.  We shall make essential use of this in describing the reverse construction, next.

\no{\em How to convert a matrix system to a scalar system.} 

Given a system 
\[
\b_i Y = A_i Y, \quad 1\le i\le r
\]
of
first order matrix p.d.e. whose coefficients are holomorphic on $N$, it is possible to construct a system of higher order scalar p.d.e., providing the connection $d-A$ corresponding to the matrix system is flat.

\no{\em Step 1:} We begin with the D-module 
$\calN=\Map(N,\C^{s+1})=\calH^{s+1}$, where the action of 
$\b_i$ is given by $\b_i\cdot Y=(\b_i-A_i)Y$.  (This extends to an action of $D$ because the flatness condition --- i.e.\  (3) of Theorem \ref{flat} --- ensures that $\b_i\cdot(\b_j\cdot Y)=\b_j\cdot(\b_i\cdot Y)$.)

\no{\em Step 2:} The dual D-module $\calN^\ast$ is defined by $\calN^\ast=\Hom_\calH(\calN,\calH)$ with 
action of $\b_i$ given by 
$\b_i\cdot p(Y)=-p(\b_i\cdot Y)+\b p(Y)/\b z_i$.  

\no{\em Step 3:} Choose\footnote{To guarantee the existence of a cyclic vector, it is necessary to enlarge $\calH$ in step 1, for example to the field of meromorphic functions on $N$. A proof can be found in \cite{PuSi03} (Proposition 2.9 and Lemma D.5). If the coefficients are polynomial functions, it suffices to replace $\calH$ by the algebra of polynomial functions (Theorem 8.18 and Corollary 8.19 of \cite{Bj79}).}
a cyclic element of $\calN^\ast$, namely an element $p_{\cyclic}$ such that $D\cdot p_{\cyclic}=\calN^\ast$.  

\no{\em Step 4:} It follows that
$\calN^\ast \cong D/I$, where $I$ is the (left) ideal of operators which 
annihilate $p_{\cyclic}$.

\no{\em Step 5:} Choose\footnote{To guarantee that $D$ is Noetherian (so that a finite set of generators of $I$ always exists), it is necessary to replace $\calH$ by, for example, the ring of holomorphic functions on a closed poly-disk (section 4 of \cite{Ph79}) or the ring of polynomial functions (section 3 of the Introduction of \cite{Bj79}).
}
generators $T_1,\dots,T_u$ for the ideal $I$.
Then a suitable scalar system (not unique) is $T_1y=0, \dots, T_uy=0$.  

We illustrate the procedure with the following (artificial) example.  
In situations which arise from geometry, cyclic elements (step 3)  and generators for ideals (step 5) often arise naturally.  

\begin{example}
Consider the matrix system
\[
\b
\bp y_0\\y_1\ep =
\bp 0 & u\\ v & 0 \ep
\bp y_0\\y_1\ep,
\]
where $u,v$ are given functions of $z$, holomorphic on $N$.
As a candidate for a cyclic element of the dual D-module we try $p_0$, defined by $p_0(Y)=y_0$.  
Since
\begin{align*}
\b\cdot p_0(Y)&=- p_0(\b\cdot Y)+ p_0(Y)^\pr\\
&=- p_0
\left(
\bp y_0\\ y_1 \ep^\pr -
\bp 0 & u\\ v & 0 \ep
\bp y_0\\y_1\ep
\right)
+y_0^\pr\\
&= -(y_0^\pr - uy_1) + y_0^\pr\\
&=uy_1\\
&=u p_1(Y),
\end{align*}
we have $\b\cdot p_0 = u p_1$ (where $p_1$ is defined by $p_1(Y)=y_1$).  If $u$ is never zero on $N$,
$p_0$ and $\b\cdot p_0$ span the D-module over $\calH$, so $ p_0$ is a cyclic element.  A similar calculation gives 
$\b^2\cdot p_0 = uv p_0 +u^\pr p_1=
uv p_0 +(u^\pr/u)\b\cdot p_0$.  
We obtain
$(\b^2 -  (u^\pr/u) \b - uv)\cdot p_0 = 0$,
so  the scalar o.d.e.\  is $(\b^2 -  (u^\pr/u) \b - uv)y=0$. 
(This computation amounts to \ll declaring that $y=y_0$\rr and computing the scalar system for $y$ from the matrix system for $Y$.)   If $1/u$ does not belong to $\calH$, we must either enlarge $\calH$ or try another candidate for a cyclic element.
\qed
\end{example}

Although the D-module  $\calM=D/(T_1,\dots,T_u)$ is the fundamental object, we can regard $\calH^{s+1}$ (with D-module structure given by $d+\Om$) as a concrete representation. This representation is often useful for calculations. Moreover, $H$ can be regarded as a gauge transformation which converts $d+\Om$ to the trivial connection $d$.  To express these correspondences it is convenient to introduce the following \ll$J$-function\rr (a name introduced by Givental in the context of quantum cohomology).

\begin{notation} Let $J=(y_{(0)},\dots,y_{(s)})$, 
where $y_{(0)},\dots,y_{(s)}$ is any basis of solutions of the scalar system
$T_jy=0$, $1\le j\le u$.
\end{notation}

\no We obtain a basis $Y_{(0)},\dots,Y_{(s)}$ of solutions of the matrix system, whose fundamental solution matrix can be written
\[
H=
\bp
\vert &  & \vert \\
Y_{(0)} & \cdots & Y_{(s)}  \\
\vert &  & \vert
\ep
=
\bp
\text{---\,} & P_0 J& \text{---} \\
 & \vdots &  \\
\text{---\,} & P_s J & \text{---}
\ep.
\]
It is usually possible to take $P_0=1$,  in which case the top row of the last matrix is just $J$.

The identifications just described are shown below:
\[
\begin{array}{ccccc}
\b_i  &  &  \b_i+\Om_i  &  &  \b_i  \\
\\
\calM &   \xrightarrow{\ \ [P_0],\dots,[P_s]\ \ } &
\calH^{s+1} & \xrightarrow{\ \ H\ \ } & \calH^{s+1} \\
\\
P=\smallsum f_i P_i  &  &  
(f_0,\dots,f_s)^t
& & PJ 
\end{array}
\]
In each column, the operator (top) acts on elements (bottom) of the D-module (middle) in the natural way.

\section{The quantum differential equations}\label{qde}

We begin by summarizing briefly the notation from cohomology theory that we shall use.  

Let $M$ be a connected simply connected compact  K\"ahler manifold,
of complex dimension $n$. {\em For simplicity we assume that the nonzero  integral cohomology groups
of $M$ are even-dimensional and torsion-free.}
We generally use lower-case letters $a,b,c,\dots \,\in H^\ast(M;\Z)$
for cohomology classes,  and upper-case letters
$A=\PD(a), \ B=\PD(b),\  C=\PD(c),
\dots\,\in\  H_\ast(M;\Z)$
for their Poincar\'e dual homology classes.  We often refer to $A,B,C,\dots$ as though they are submanifolds or subvarieties of $M$ (rather than equivalence classes of cycles). We often regard $a,b,c,\dots$ as differential forms on $M$. 

Let
$\lan\ ,\ \ran:H^\ast(M;\Z) \times H_\ast(M;\Z) \to \Z$
denote  the natural nondegenerate pairing.  In de Rham notation, $\lan a,B \ran=\int_B a$. The intersection pairing is defined by
\[
(\ ,\ ):H^\ast(M;\Z) \times H^\ast(M;\Z) \to \Z,\quad
(a,b)=\lan ab, M\ran = \int_M a\wedge b.
\]
We have
$\lan ab, M\ran = \lan a,B\ran = \lan b,A \ran$.
This is a nondegenerate symmetric bilinear form.   

We are interested primarily in the cup product operation on cohomology, and its generalization to the quantum product.  It is convenient for this purpose to specify the cup product by giving
its \ll structure constants\rr with respect to a basis. We can choose bases as follows
\[
H_\ast(M;\Z) = \bigoplus_{i=0}^s \Z A_i, \quad
H^\ast(M;\Z) = \bigoplus_{i=0}^s \Z a_i
\]
and then define dual cohomology classes
$b_0,\dots,b_s$ 
by $(a_i,b_j)= \de_{ij}$. Then for any $i,j$ we have
\[
a_ia_j=\sum_{i,j,k} \la_{ijk} b_{k}
\]
where
\[
\la_{ijk}= \lan a_ia_ja_k, M\ran=
\int_M a_i\wedge a_j\wedge a_k =
\sharp\  A_i\cap A_j \cap  A_k.
\]
Note that the intersection form itself can be specified in a similar way by the integers
\[
(a_i,a_j)=
\int_M a_i\wedge a_j =
\sharp\  A_i\cap A_j.
\]
It is modern practice to regard a cohomology theory as a functor from a certain category of topological spaces to the category of groups which satisfies the Eilenberg-Steenrod axioms.  However, from the point of view of quantum cohomology, which we shall consider next, it is preferable to regard the cohomology of $M$ as a collection of numbers
\[
\sharp\,  A_i\cap A_j,\quad \sharp\,  A_i\cap A_j \cap  A_k.
\]
This primitive viewpoint is necessary because quantum cohomology does not (at present) have a functorial characterization.

If we denote $\sharp\,  A_i\cap A_j \cap  A_k$ by
$\lan A_i\vert A_j\vert A_k\ran_0$, the quantum product is obtained by extending the above collection of numbers to an infinite sequence of \ll Gromov-Witten invariants\rr $\lan A_i\vert A_j\vert A_k\ran\subD$ for any $D\in \pi_2(M)\cong H_2(M;\Z)$, as follows.
Let $p,q,r$ be three distinct points in $\C P^1$.  
Informally, the definition is
\[
\ABC\subD=\sharp\ 
\HolD^{ A,p} \cap \HolD^{ B,q} \cap \HolD^{ C,r}
\]
where
\[
\HolD^{ A,p}=
\{ \textrm{holomorphic maps}\ f:\C P^1\to M \st
f(p)\in A\ \textrm{and}\  [f]=D \},
\]
and $[f]$ is the homotopy class of $f$. 

As explained in \cite{CoKa99} (for example), $\ABC\subD$ can be defined rigorously under very general conditions.  The definition
has the form
\[
\ABC\subD\ = \int_{[\overline{M}(D)]^{\text{virt}}} 
ev_1^\ast a\wedge ev_2^\ast b\wedge ev_3^\ast c
\]
where $M(D)$ is a certain moduli space of \ll curves\rrr,
$\overline{M}(D)$ is a compactification of $M(D)$, obtained by adding suitable \ll boundary components\rrr, and
$[\overline{M}(D)]^{\text{virt}}$ denotes the \ll virtual fundamental class\rr over which integration is carried out.
The evaluation map $ev_i:\overline{M}(D)\to M$ assigns to a curve its value at a given $i$-th basepoint ($i=1,2,3$).  

To define $a\circ_t b$ for 
$a,b\in H^\ast M$ and $t\in H^2 M$, it suffices
to define $\lan a\circ_t b, C\ran$ for all $C\in H_\ast M$.
The definition is:

\begin{definition} Assume that $M$ is a Fano manifold.  Then the quantum product 
$a\circ_t b$ of two cohomology classes $a,b\in H^\ast M$ is
defined by
\[
\lan a\circ_t b, C\ran=\sum_{D\in H_2(M;\Z)}\,\ABC\subD\, e^{\lan t,D\ran}.
\]
\end{definition}

\no The Fano condition ensures that the sum is finite, and in this case one has the following nontrivial theorem (see section 8.1 of \cite{CoKa99}):

\begin{theorem}\label{product}  For each $t\in H^2 M$,  $\circ_t$ is a commutative, associative product operation on $H^\ast M$.
\end{theorem}

\no In general the quantum product is supercommutative, but it is commutative here as we are assuming that the odd-dimensional cohomology of $M$ is zero.
We denote the algebra $(H^\ast M, \circ_t)$ (more precisely, family of algebras) by $QH^\ast M$, and refer to it as the quantum cohomology algebra of $M$.  

Since the second cohomology group plays a prominent role in quantum cohomology,  we shall assume that the basis $A_0,\dots,A_s$ has been chosen so that $A_1,\dots,A_r$ span $H_2 M$ and $b_1,\dots,b_r$ span $H^2 M$.
A general element of $H^2M$ will be written
$t=\sum_1^r t_i b_i\in H^2M$; the Poincar\'e dual homology class is then $T=\sum_1^r t_i A_i$.
It is conventional to introduce the notation $q_i=e^{t_i}$.  However, $\b_i$ always denotes the derivative with respect to $t_i$; thus $\b_i=q_i\b/\b q_i$.

\begin{example}  The standard basis of $H^\ast \C P^n$ is
$1,b,b^2,\dots,b^n$ where $b$ is the (Poincar\'e dual of the) hyperplane class. We take this basis as $b_0,\dots,b_n$. A well known calculation 
(see section 8.1 of \cite{CoKa99}) gives the quantum products
\[
b_i\circ_t b_j =
\begin{cases}
b_{i+j}\quad\textrm{if}\ 0\le i+j\le n\\
b_{i+j-(n+1)}q\quad\textrm{if}\ n+1\le i+j\le 2n.
\end{cases}
\]
In particular we obtain the presentation 
$QH^\ast \C P^n \cong
\C[b,q]/( b^{n+1}-q)$, in which  $q$ is regarded as a formal parameter, rather than the number $e^t$. {\em  In this article we shall switch between these two versions of quantum cohomology without further comment.}
\qed
\end{example} 

\begin{example}\label{hyper} 
Let $M=M^k_N$ be a nonsingular complex hypersurface of degree $k$ in $\C P^{N-1}$.  All such hypersurfaces have the same cohomology algebra. 
The Lefschetz Theorems show that $H_i M^k_N\cong H_i\C P^{N-1}$ for $0\le i\le 2N-4$ except possibly for the middle dimension $i=N-2$, and that the subalgebra $H^\sharp M^k_N$ generated by $H^2 M^k_N$ has additive generators represented by cycles of the form $M^k_N\cap\C P^j$.  To avoid odd-dimensional cohomology, and make use of the above cycles, we shall restrict attention to the subalgebra $H^\sharp M^k_N$  and its quantum version $QH^\sharp M^k_N$.  

Let us write $b=b_1$ for the hyperplane class, i.e.\ the cohomology class Poincar\'e dual to $M^k_N\cap\C P^{N-2}$. We have $c_1(TM^k_N)=(N-k)b$.  It follows that $M^k_N$ is Fano if and only if $1\le k\le N-1$.  
The classes $1,b,\dots,b^{N-2}$ are an additive basis (over $\C$) for $H^\sharp M^k_N=H^\sharp (M^k_N;\C)$, and the intersection form is given by $(b^i,b^j)= k\de_{i+j,N-2}$.

As a concrete example, we shall just give the quantum products for $M^3_5$.  All quantum products in this case follow from
\begin{align*}
b\circ_t 1\,\,&=b\\
b\circ_t \,b\,\,&= b^2 + 6q\\
b\circ_t b^2&= b^3 + 15qb\\
b\circ_t b^3&= 6qb^2+36q^2
\end{align*}
(see \cite{CoJi99}, \cite{Ji02}).
In particular $b\circ_t b\circ_t b=b^3 + 21qb$ and $b\circ_t b\circ_t b\circ_t b=27qb\circ_t b$. We deduce that
$QH^\sharp M^3_5\cong 
{\C[b,q]}/{(b^4-27qb^2)}$.
\qed
\end{example}

As we have mentioned, the construction  $M \mapsto QH^\ast M$ is, unfortunately,  not functorial. This is perhaps not surprising in view of the fact that the Gromov-Witten invariants $\ABC\subD$ contain much more information than the isomorphism class of the algebra $QH^\ast M$. Therefore we are led to consider other objects constructed from the Gromov-Witten invariants, and the most prominent of these is the quantum D-module $\calM$ (see  \cite{Gi95-1}, \cite{Gi95-2}).  

Let us consider the space of sections of the trivial vector bundle
\[
H^2M \times H^\ast M \to H^2M
\]
or, more generally, the space of sections over an open subset $N$ of $H^2 M$.  This is just the vector  space consisting of all $H^\ast M$-valued functions on $N$.
The quantum product $\circ_t$ on $H^\ast M$ gives a way of multiplying sections.  Thus the space of sections becomes an algebra over $\calH_t$. 

Next,  we introduce the action  of a ring of differential operators on sections, i.e.\ a D-module structure.  We do this by defining a connection $\nabla$, called the Dubrovin connection or Givental connection.   The definition is:
$\nabla_{\b_i}=\b_i+\tfrac1\h b_i\circ_t$, for $1\le i\le r$, where $\h$ is a parameter.  If $\om_i$ is the matrix of quantum multiplication by $b_i$ with respect to the basis $b_0,\dots,b_s$, then we can also write $\nabla_{\b_i}=\b_i + \frac1\h\om_i$.  
The D-module structure is specified by saying that  $\b_i$ acts as  $\nabla_{\b_i}$.
This extends to an action of the ring of all differential operators if the identity $\nabla_{\b_i}\nabla_{\b_j}=\nabla_{\b_j}\nabla_{\b_i}$ holds for all $i,j$, and this 
identity does hold because the connection is flat --- a consequence of the properties of the quantum product (see section 8.5 of \cite{CoKa99}).

It is convenient to incorporate the parameter $\h$ into the ring of differential operators. Thus, we shall take as  ring of differential operators the ring $D^\h$ which is generated by $\h\b_1,\dots,\h\b_r$ and whose elements have coefficients which are holomorphic in $t\in N$ and holomorphic in $\h$ in a neighbourhood of $\h=0$.  This acts on the enlarged space of sections, in which the sections are allowed also to depend on $\h$ (holomorphically,  in a neighbourhood of $\h=0$). The quantum D-module $\calM$ is defined to be this enlarged space of sections.  (We use the generic term \ll quantum D-module\rrr, rather than \ll quantum $D^\h$-module\rrr, for simplicity, but $\calM$ is of course a module over the ring $D^\h$.)

The most important property of the quantum D-module is its close relation with the quantum cohomology algebra $QH^\ast M$.  We shall discuss the relation in this section {\em under the assumption that $H^2 M$ generates $H^\ast M$ as an algebra and $M$ is a Fano manifold.}
These hypotheses imply that $QH^\ast M$ has a presentation
 \[
QH^\ast M=
\C[b_1,\dots,b_r,q_1,\dots,q_r]/
(\calR_1,\dots,\calR_u )
\]
and $H^\ast M$ has a presentation
\[
H^\ast(M;\C)=
\C[b_1,\dots,b_r]/(R_1,\dots,R_u )
\]
where $\calR_i\vert_{q=0}=R_i$.  However, there is a more precise connection between $QH^\ast M$ and $H^\ast M$, which generalizes to a precise connection between $\calM$ and $QH^\ast M$, so let us review this.

First, for any polynomial $c$ in \ll abstract variables\rr $b_1,\dots,b_r,q_1,\dots,q_r$, let us
denote by $[c]$ the corresponding element of $QH^\ast M$, and by $[[c]]$ the corresponding element of $H^\ast M\otimes\C[q_1,\dots,q_r]$. We claim that there exist
suitable polynomials $c_0,\dots,c_s$ such that
\begin{align*}
[b_i][c_j]&=\sum_{k=0}^s (\om_i)_{kj} [c_k]\\
[[b_i]]\circ_t[[c_j\vert_{q=0}]]&=\sum_{k=0}^s (\om_i)_{kj} [[c_k\vert_{q=0}]]
\end{align*}
for $1\le i\le r$.  In other words, if we identify 
$QH^\ast M$ with $H^\ast M\otimes\C[q_1,\dots,q_r]$ via the bases given by $c_0,\dots,c_s$
 and  $c_0\vert_{q=0},\dots,c_s\vert_{q=0}$, then 
quantum multiplication in $H^\ast M$ corresponds to the natural multiplication in $QH^\ast M$.
This follows from the observation (Theorem 2.2 of \cite{SiTi97}) that \ll any quantum polynomial may be written as the same classical polynomial plus lower classical terms, and vice versa\rrr.  Namely, if we regard $b_j$ as a polynomial (with respect to the cup product) in 
$b_1,\dots,b_r$,   then the polynomial $c_j$ is obtained by expressing $b_j$ as a polynomial with respect to the quantum product in 
$b_1,\dots,b_r$.  The polynomials $c_j$ satisfy $c_j\vert_{q=0}=b_j$.  

Exactly the same method gives the analogous result below for $\calM$, because any polynomial in the operators
$\h\b_1+\om_1,\dots,\h\b_r+\om_r$ can be expressed as the same polynomial in $\h\b_1,\dots,\h\b_r$ plus terms of lower order. Moreover, since the lower order terms contain \ll additional\rr powers of $\h$, if we replace $\h\b_i$  by $b_i$ (for each $i$)  then $\h$  set equal to $0$,  these lower order terms all vanish and we are left with the original polynomial expressed in terms of the variables $b_1,\dots,b_r$.

\begin{theorem}\label{match}  The quantum D-module is  isomorphic to a D-module of the form
$D^\h/(D_1,\dots,D_u)$,
where $D_1,\dots,D_u$ are converted to $\calR_1,\dots,\calR_u$ when $\h\b_i$ is replaced by $b_i$ (for each $i$)  then $\h$  set equal to $0$. 
Furthermore, there exists a  basis $[P_0],\dots,[P_s]$ of 
$D^\h/(D_1,\dots,D_u)$, with respect to which the (connection) matrix of $\h\b_i$ is $\om_i$. This basis is converted to $[c_0],\dots,[c_s]$ when $\h\b_i$ is replaced by $b_i$ (for each $i$) then $\h$ set equal to $0$. 
\end{theorem}

The basis $[P_0],\dots,[P_s]$ gives a correspondence between a scalar system and a matrix system (cf.\ section \ref{pde}), both of which are referred to as the quantum differential equations.  A particular choice of $J$-function (see section 5.2 of \cite{Gu08} and the original paper \cite{Gi96}) has the remarkable property that it can be written explicitly in terms of Gromov-Witten invariants. 

\begin{example}\label{hyper2}  Let us continue Example \ref{hyper} by finding a relation $D_1$ and a basis $[P_0],[P_1],[P_2],[P_3]$  as predicted by the above theorem.  
(The computation of $D_1$ is analogous to steps 1-5 in section \ref{pde}; it would be possible to obtain a version of Theorem \ref{match} for more general modules over $D^\h$ this way.)
From the quantum products, the D-module structure is given by
\[
\h\b\cdot
\bp
f_0\\ f_1\\ f_2\\ f_3
\ep
=
(\h\b+\om)
\bp
f_0\\ f_1\\ f_2\\ f_3
\ep,
\quad
\om=
\bp 
 &6q & & 36q^2 \\
 1 & &15q & \\
  & 1 & & 6q \\
   & & 1 & 
\ep.
\]
Writing $e_0,e_1,e_2,e_3$ for the standard basis of column vectors, repeated application of $\h\b$ gives:
\begin{align*}
(\h\b)^1\cdot e_0&=e_1\\
(\h\b)^2\cdot e_0&=6qe_0+e_2\\
(\h\b)^3\cdot e_0&=6\h q e_0+21qe_1+e_3.
\end{align*}
This shows that $e_0$ is a cyclic element.
Now we \ll solve\rr for $e_0,e_1,e_2,e_3$, to obtain:
\begin{align*}
e_1&=\ \h\b\cdot e_0\\
e_2&=\left((\h\b)^2-6q\right)\cdot e_0\\
e_3&=\left(
(\h\b)^3-21q\h\b-6\h q
\right)\cdot e_0.
\end{align*}
It follows that the matrix of $\h\b$ with respect to 
$[1]$, $[\h\b]$, $[(\h\b)^2-6q]$,
$[(\h\b)^3-21q\h\b-6\h q]$ is the above matrix $\om$, so
this is the required basis.  

To obtain a relation for the D-module, i.e.\ a differential operator which annihilates the cyclic element, we differentiate once more:
\[
(\h\b)^4\cdot e_0=
162q^2e_0 + 27qe_2 + 27\h q e_1 +  6\h^2 q e_0
\]
Substituting for $e_1,e_2,e_3$, we obtain
\[
\left(
(\h\b)^4 - 27q(\h\b)^2 -27\h q(\h\b) - 6\h^2q
\right)\cdot e_0=0.
\]
We conclude that
$\calM^{M^3_5}=D^\h/\left(
(\h\b)^4 - 27q(\h\b)^2 -27\h q(\h\b) - 6\h^2q
\right)$.

At the commutative level, i.e.\ in the quantum cohomology algebra, the analogous calculation would give
\begin{align*}
b^1\cdot 1&=b\\
b^2\cdot 1&=6q+b^2\\
b^3\cdot 1&=21qb
\end{align*}
(these were stated at the end of Example \ref{hyper}; the notation $\,b^i\cdot\,$ here indicates the $i$-fold iteration of $b\circ_t$). Then
\begin{align*}
b^1&=\ b\cdot 1\\
b^2&=\left(b^2-6q\right)\cdot 1\\
b^3&=\left(
b^3-21qb
\right)\cdot 1.
\end{align*}
Thus we obtain $c_0=1$, $c_1=b$, $c_2=b^2-6q$,
$c_3=b^3-21qb$.   Applying $b$ again leads to the relation $b^4-27qb^2$, as expected.
\qed
\end{example}

The replacement of $QH^\ast M$ by $\calM$ is not unlike the process of quantization in physics. In the above
example the relation $b^4-27qb^2$ is \ll quantized\rr to the relation $(\h\b)^4 - 27q(\h\b)^2 -27\h q(\h\b) - 6\h^2q$.    For $\C P^n$ the same argument shows that the relation $b^{n+1}-q$ is converted to the naive quantization $(\h\b)^{n+1}-q$.  However, the naive quantization does not always work: there are examples where the \ll naive quantization\rr is not a quantization at all, because it gives a D-module of the wrong rank.  (For the example above, the naive quantization $(\h\b)^4 - 27q(\h\b)^2$ gives the correct rank, but it gives the wrong quantum products. We shall return to this point in section \ref{conclusion}.) 

In general, the parameter $\h$ keeps track of the difference between the commutative and noncommutative multiplications.
The incompatibility of the commutative and noncommutative situations reveals the key property of the quantum D-module. Namely,  the action of $\h\b_i$ on the quantum D-module  matches exactly the action of $b_i$ on the quantum cohomology algebra --- both are given by the same matrix.  However, to accomplish this, a careful choice of basis is necessary in each case, and these bases (like the relations) do not match exactly, only \ll mod $\h$\rrr.  If the bases did match exactly (for example, $(\h\b)^i$ and $b^i$), then the matrices of $\h\b_i$ and $b_i$ would not in general be the same.

In the above example, the basis  $[P_0],[P_1],[P_2],[P_3]$ was produced by modifying the monomial basis $[1], [\h\b], [(\h\b)^2], [\h\b)^3]$, and this involved solving a system of linear equations --- by Gaussian elimination. Gaussian elimination may be described as the process of finding a  lower triangular/upper triangular factorization of matrices.  It turns out that such a modification is always possible (under our 
assumption that $H^2 M$ generates $H^\ast M$ and $M$ is Fano),  by using a suitable factorization which takes account of the parameter $\h$, 
known as
the Birkhoff factorization, from \cite{PrSe86}.  This says that \ll almost every\rr loop $\ga\in \La \glsc$ (i.e.\ almost every smooth map $\ga:S^1\to \glsc$)
may be factorized in the form
\[
\ga(\h) = 
\underbrace{
(a_0+\tfrac1\h a_1 + \tfrac1{\h^2} a_2+\cdots)}_{\textstyle\ga_-(\h)}
\underbrace{
(b_0+\h b_1 + \h^2 b_2+\cdots)}_{\textstyle\ga_+(\h)}.
\]
The subgroup of $\La \glsc$ consisting of \ll negative\rr (\ll positive\rrr) loops is denoted $\La_-\glsc$ ($\La_+\glsc$). The meaning of \ll almost every\rr is that the product set
$\La_-\glsc\, \La_+\glsc$ is open and dense in $\La\glsc$.

\begin{theorem}\label{basis}
Assume that  
$M$ is Fano, and that $H^2 M$ generates $H^\ast M$.
 Let $L^\mon$ be any solution of 
$(L^\mon)^{-1}dL^\mon=\Om^\mon$, where
$\Om_i^\mon$ is the matrix of $\h\b_i$ with respect to a monomial\footnote{If $\dim H^2 M=1$ and $\dim H^\ast M=s+1$, a monomial basis is 
$[1], [\h\b],\dots,[(\h\b)^s]$. The meaning of monomial basis in general is explained in section 6.6 of \cite{Gu08}.}
 basis $[P_0^\mon],\dots,[P_s^\mon]$. 
 Let $L^\mon=L_-^\mon L_+^\mon$ be the\footnote{It can be shown that this factorization is possible in a punctured neighbourhood of $q=0$, if $L_-^\mon$ is allowed to be multiple-valued. See sections 5.3 and 5.4 of \cite{Gu08}.}
Birkhoff factorization.  Then the Dubrovin/Givental connection is given by $\Om=L^{-1}dL$ where $L=L_-^\mon$.  A basis 
$[P_0],\dots,[P_s]$ of the type predicted in Theorem
\ref{match}  is given by
$(L_+^\mon)^{-1}\cdot P_0^\mon, 
(L_+^\mon)^{-1}
\cdot P_1^\mon, \dots,
(L_+^\mon)^{-1} \cdot P_s^\mon$, where  $(L_+^\mon)^{-1}\cdot P_i^\mon$ means $\sum_{j=0}^{s} (L_+^\mon)^{-1}_{ji} P_j^\mon$.
\end{theorem}

For this we refer to \cite{Gu05} and section 6.6 of \cite{Gu08}, where it is also shown that $L_+^\mon$ is of the form $L_+^\mon=Q_0(I+\h Q_1+\h^2 Q_2+\cdots+\h^N Q_N)$, i.e.\ a finite series in $\h$, and that the coefficient matrices $Q_0,Q_1,\dots,Q_N$ may be found by a simple algorithm.  
The advantage of $\calM$ over $QH^\ast M$ is that it contains all the Gromov-Witten invariants, and this algorithm shows how to extract them.  

\begin{example}\label{hyper3}  
We state the results of applying this algorithm to Example \ref{hyper2}, i.e.\  $QH^\sharp M^3_5$.     
First, with respect to the monomial basis $[1], [\h\b], [(\h\b)^2], [(\h\b)^3]$ the connection matrix is
\[
\Om^\mon=
\tfrac1\h
\bp 
 & & & 6q\h^2 \\
 1 & & & 27q\h \\
  & 1 & & 27q \\
   & & 1 & 
\ep.
\]
Then it turns out that $L_+^\mon=Q_0(I+\h Q_1)$, with
\[
Q_0=
\bp 
\, 1\,  & &6q &  \\
  & \,1\,& & 21q \\
  &  & 1&  \\
   & &  & 1
\ep,
\quad
Q_1=
\bp 
\ \  & \ \ & \ \ &  6q\\
  & & &  \\
  &  & &  \\
   & &  & 
\ep.
\]
Computing $(L_+^\mon)^{-1}\cdot (\h\b)^i$ gives the  answer $P_0=1, P_1= \h\b, P_2= (\h\b)^2-6q,
P_3= (\h\b)^3 - 21q\h\b - 6\h q$ that we obtained in Example  \ref{hyper2}.
\qed
\end{example}

\section{A D-module construction of integrable systems}\label{extensions}

An {\em integrable p.d.e.\ {\rm is}} a p.d.e.\ which can be written as a zero curvature condition $d\Om + \Om\wedge\Om=0$, where $\Om$ is given in terms of some auxiliary function(s) $u=u(z_1,\dots,z_r)$.  This concept is somewhat related to the \ll explicit solvability\rr of the p.d.e., and closely related to the concept of \ll integrable system\rrr.  
It is easy to write down connection forms $\Om$ which depend on auxiliary functions, then compute the condition $d\Om + \Om\wedge\Om=0$. However,  it is not easy to produce {\em nontrivial} examples this way.  In terms of D-modules, a random choice of ideal $I$ generally leads to a D-module $D/I$ of rank infinity or rank zero. 

We have seen that quantum cohomology leads to D-modules of finite rank.  It is natural to ask whether the quantum cohomology of a particular space can be regarded as a solution of an integrable p.d.e., and whether more general \ll quantum cohomology like\rr finite rank D-modules can be constructed.
Let us begin with two simple examples.  In this section, $D$ denotes $D_{z_1,\dots,z_r}$, but we omit $z_1,\dots,z_r$ when there is no danger of confusion.

\begin{example}\label{rankone} Let 
\[
T_1=\b_1+f,\quad T_2=\b_2+g
\]
where $f$ and $g$ are functions of $z_1,z_2$.  Clearly the rank of $D/(T_1,T_2)$ is either $1$ or $0$, since we have
$[\b_1]=-f[1]$, $[\b_2]=-g[1]$. The only question is whether $[1]=[0]$ in the D-module, and this depends on $f$ and $g$.  Alternatively, the rank is $1$ if and only if the solution space of the linear system 
$(\b_1+f)y=0, (\b_2+g)y=0$ has dimension $1$. The condition for this is $\b f/\b z_2 = \b g/\b z_1$, which can be regarded as a p.d.e.\ for the functions $f,g$. 
\qed
\end{example}

\begin{example}\label{KdV}
Let
\[
T_1=\b_1^2+u,\quad
T_2=\b_2-
(\tfrac12 u_{z_1}-u\b_1)
\] 
where $u$ is a given function of $z_1,z_2$.
It is clear that the rank of $D/(T_1,T_2)$ is at most $2$, because
$\b_2$ is expressed in terms of $\b_1$ and $T_1$ is quadratic
in $\b_1$.  Whether the rank is exactly $2$ depends on whether  $[1]$ and $[\b_1]$ are independent, and this depends on the nature of $u$.  It can be shown that the
rank is $2$ if and only if $u$ satisfies the condition
$
u_{z_2}=3uu_{z_1} + \tfrac12 u_{z_1z_1z_1},
$
which is the
KdV equation.
\qed
\end{example}

Both of these examples arise in the following way: first, we fix a value of $z_2$, and consider the single variable D-module $D/(T_1)$, whose rank is obvious (the order of $T_1$); then, we attempt to \ll extend\rr to a two variable D-module of the same rank by adding a relation of the form $T_2=\b_2-P$.  

Let us make this into a general procedure. We call the variables $x$ and $t$, as the procedure can be interpreted as producing a $t$-flow of the original D-module in $x$.  Thus, we start with a D-module $D_x/(T)$ of rank $s+1$, where 
\[
T=\bx^{s+1} + a_s(x)\bx^s + \dots + a_1(x)\bx + a_0(x).
\]
We wish to extend this to a D-module $D_{x,t}/(T_1,T_2)$ of rank $s+1$ by extending $T$ to a $t$-family $T_1$ (with $T_1\vert_{t=0}=T$)  and adjoining a further partial differential operator $T_2$.
If we take $T_2$ of the form $T_2=\bt - P$, where $P$ does not involve $\bt$, then it is obvious that 
\[
\rank \ D_{x,t}/(T_1,T_2) \ \le\  s+1
\]
since $T_2$ may be used to eliminate $\bt$.  

\begin{proposition}\label{rank} Let $T_2=\bt-P$. Then
$\rank \ D_{x,t}/(T_1,T_2) = s+1$ if and only if any of the following equivalent conditions hold:

\no(a) $[T_2,(T_1)]\sub (T_1)$

\no(b) $[T_2,T_1] \equiv 0 \mod T_1$

\no(c) $(T_1)_t \equiv[P,T_1]\mod T_1$

\no(where $(T_1)_t$ means the result of differentiating the coefficients of $T_1$ with respect to $t$).
\end{proposition} 

\begin{proof}[Sketch proof]  The proof hinges on the construction of a certain connection in the trivial bundle $N_{x,t}\times \C^{s+1}$.  We shall define the connection form with respect to the local basis $[1],[\bx],\dots,[\bx^s]$.   First of all, the $t$-family of D-modules $D_x/(T_1)$ gives a connection $\nabla_{\bx}$  in the $x$-direction,
namely $\nabla_{\bx} [\bx^i]=[\bx^{i+1}]$.
Next we define $\nabla_{\bt}$ by
$\nabla_{\bt} [\bx^i]=[\bx^{i}P]$ for $0\le i\le s$.  
We claim that the resulting connection $\nabla$ is flat if and only if condition (c) holds.  Details can be found in section 4.4 of \cite{Gu08}.
\end{proof}

The proof suggests a useful computational method: first write down the connection form $\Om=\Om_1 dx + \Om_2 dt$ with respect to $[1],[\bx],\dots,[\bx^s]$, then calculate $d\Om +\Om\wedge\Om$.  Let us apply this to the two examples given earlier.  

For Example \ref{rankone} the connection form is just  $\Om=(-f)dz_1 + (-g)dz_2$.  Since we are dealing with $1\times1$ matrices, we have $\Om\wedge\Om=0$, so the flatness condition is $d\Om=0$, i.e.\ $f_{z_2}=g_{z_1}$.
For Example \ref{KdV}, let us consider a more general relation
$T_2=\bt-P$ where $P=f+g\b_x$ (keeping $T_1=\bx^2+u$).  
To find the connection matrix of $\nabla_{\bx}$, we compute
\begin{align*}
\bx[1]&=[\bx]=0[1]+1[\bx]\\
\bx[\bx]&=[\bx^2]=[-u]=-u[1]+0[\bx],
\end{align*}
so
\[
\Om_1=
\bp
  0  &   -u \\
1 & 0
\ep.
\]
Similarly, from
\begin{align*}
\bt[1]&=[\bt]=[P]=f[1]+g[\bx]\\
\bt[\bx]&=[\bx P]=[\bx(f+g\bx)]=
(f_x-ug)[1]+(f+g_x)[\bx],
\end{align*}
we obtain
\[
\Om_2=
\bp
  f  &  f_x -ug \\
g & f+g_x
\ep.
\]
Now, the zero curvature condition
$d\Om+\Om\wedge\Om=0$ reduces to
\begin{align*}
-u_t&=f_{xx}-u_x g - 2ug_x\\
0&=2f_x+g_{xx},
\end{align*}
hence
\[
u_t=\tfrac12 g_{xxx}  + g u_x + 2g_x u.
\]
To obtain an evolution equation such as the KdV equation it is natural to take $f$ and $g$ to be differential polynomials in $u$ (it suffices to choose $g$, as without loss of generality $f=-\tfrac12 g_x$).  The choice $g=-u$ (giving $P=\tfrac12u_x-u\bx$) produces the KdV equation of Example \ref{KdV}.  
There are many other choices, and the possibilities multiply further when we start with a general operator $T_1$ instead of $T_1=\bx^2+u$.   Thus our construction produces a vast number of examples of \ll integrable p.d.e.\rrr.   

It is necessary to make a remark here about the special role of KdV equation, which is more commonly 
viewed as the Lax eqation
$(T_1)_t=[P,T_1]$, with $P=\bx^3+\frac32 u\bx + \frac34 u_x$.  (In the D-module we have $P\equiv -\tfrac14 u_x+\tfrac12 u \bx$, and the Lax equation implies that
$(T_1)_t\equiv [-\tfrac14 u_x+\tfrac12 u \bx,T_1]$, so we obtain the same KdV equation from this $P$.)
The condition $[\bt-P,T_1]\equiv 0$ mod $T_1$ can be regarded as the intrinsic scalar version of the matrix zero curvature condition
$d\Om+\Om\wedge\Om=0$, but of course it is weaker than the condition
$[\bt-P,T_1]= 0$, in general.  The KdV equation is very special, as in this case the scalar version {\em can} be written in the form 
$[\bt-P,T_1]= 0$.  

The proof of Proposition \ref{rank} easily generalizes in one direction (see section 4.4 of \cite{Gu08}):

\begin{proposition}
Let $T_i$ be a $t$-family of differential operators in the variables  $z_1,\dots,z_r$ such that the D-module 
$\calM=D_{z_1,\dots,z_r}/(T_1,\dots,T_u)$ has rank $s+1$ for each value of $t$.  Let $P$ be a $t$-family of differential operators in $z_1,\dots,z_r$ such that
$[\b_t-P,\,I\,]\sub I$.
Then the extended D-module 
$D_{z_1,\dots,z_r,t}/(T_1,\dots,T_u,\b_t-P)$
also has  rank $s+1$.
\end{proposition}

\no This can be used inductively to construct \ll hierarchies\rr of integrable p.d.e., including the well known KdV hierarchy. 

Our extension procedure appears to produce very special D-modules, but it is in fact rather general.  Namely, in a  \ll generic\rr  D-module of rank $s+1$ of the form 
$D_{x,t}/I$, the elements
\[
[1],\ \b_x[1],\ \b_x^2[1],\ \dots,\ 
\b_x^s[1]
\]
will be independent.  They necessarily satisfy a relation of the form $T=\b_x^{s+1} + a_s \b_x^s + \dots + a_0$, i.e.\ 
$T [1]=0$.  The element $[\b_t]$ can be expressed as a 
linear combination of the above basis vectors,  that is, $(\b_t-P) [1]=0$ for some polynomial $P$ in $\b_x$.  Hence the D-module is of the type constructed in this section.

We conclude with a brief comment on the \ll spectral parameter\rrr. It is easy to write down a connection matrix $\Om$ with a sprinkling of $\la$'s, then obtain an \ll integrable p.d.e.\ with spectral parameter\rr $d\Om+\Om\wedge\Om=0$, but, just as when $\la$ is absent, it is not easy to produce nontrivial examples. However, such a parameter appears naturally  in many integrable systems.  
For example, the Lax form of the KdV equation is often written as $[\bt-P,T_1-\la]=0$, rather than $[\bt-P,T_1]=0$. These are equivalent, but the parameter $\la$ 
(eigenfunction of the Schr\"odinger operator $T_1$) plays an important role in describing the {\em solutions} of the KdV equation.   For the quantum differential equations we have a natural parameter $\la=\h$ from the start.  In such cases, the D-module treatment can be modified by incorporating the spectral parameter into the ring of differential operators, although some care is needed as the nature of the $\la$-dependence of the operators plays a crucial role. 

\section{Applications}\label{applications}

The main justification for the D-module language of sections \ref{pde},\ref{extensions} is that it provides a unified approach to various kinds of integrable systems with quite different geometrical interpretations.  Superficially the geometry arises from flat connections, but there are deeper undercurrents flowing between differential geometry, symplectic geometry and algebraic geometry which produce these connections.

In the case of quantum cohomology, we have already seen 
(Theorem \ref{basis}, Example \ref{hyper3}) how Gromov-Witten invariants are packaged efficiently by the quantum D-module.  It is natural to expect that properties of quantum cohomology will correspond to properties of D-modules. We shall give several examples in this direction, all of which make contact with current research.  

It is also natural to expect benefits from thinking of quantum cohomology in terms of integrable systems, and, conversely,  developing a theory of integrable systems which resemble quantum cohomology in some way.  Two key examples are the direct relation between \ll higher genus\rr quantum cohomology and the KdV hierarchy discovered by E.~Witten and M.~Kontsevich, and the classification of certain integrable systems developed by B.~Dubrovin and Y.~Zhang.  We do not discuss these here, as they primarily involve infinite hierarchies and D-modules of infinite rank. 

\subsection{The WDVV equation and reconstruction of big quantum cohomology}

It is time to address the question \ll Of which integrable system is the quantum cohomology (of a given space) a solution?\rrr.
There are two main candidates, and each of them involves a considerable digression.  

The first candidate is the WDVV equation.  This applies to \ll big quantum cohomology\rr rather than the \ll small quantum cohomology\rr that we have seen so far, but the former may be \ll reconstructed\rr from the latter and this is where the D-module extension procedure of section \ref{extensions} is relevant.

Let us briefly give the definition of big quantum cohomology and the WDVV equation.  First, the Gromov-Witten potential of a manifold $M$ is the generating function 
\[
\F^M(t)=\sum_{l\ge3,D} \tfrac1{l!} \lan
\underbrace{T\vert \dots\vert T}_l
\ran\subD
\]
for the Gromov-Witten invariants. (We assume that this function converges for $t$ in some open subset of $H^\ast M$; i.e.\ for
$T$ in some open subset of $H_\ast M$.)  It follows from this and the definition of the small quantum product that
\[
\b_i\b_j\b_k \F^M\,
\vert_{H^2M}\ (t)=
(b_i\circ_t b_j, b_k).
\]
It is natural to define a new product, called the big quantum product, as follows:

\begin{definition}  For any $t\in H^\ast M$ such that $\F^M(t)$
converges,  we define 
$\circ_t$ on $H^\ast M$ by
$( b_i\circ_t b_j, b_k)=\b_i\b_j\b_k \F^M(t)$ for all $i,j,k\in\{0,\dots,s\}$. 
\end{definition}

It can be proved (see section 8.2 of \cite{CoKa99}) that this big quantum product is commutative, associative, and has the same identity element $1$ as the small quantum product. The most difficult part of this is the associativity.

For {\em any} (smooth or analytic) $\C$-valued function
$\F$ on (an open subset of) $H^\ast M$, we can define a product 
operation $\ast_t$  in the same way:
\[
(b_i\ast_t b_j, b_k)=\b_i\b_j\b_k \F(t)
\ \overset{\text{def}}{=}\ 
\F_{ijk}(t).
\]
Whether this product is associative is a nontrivial condition on $\F$.  Commutativity is obvious.

\begin{definition}
The WDVV equation is the system of third order nonlinear partial differential equations for $\F$ given by the associativity conditions $( b_i\ast_t b_j)\ast_t b_k =
b_i\ast_t (b_j \ast_t b_k)$.
\end{definition}

\no In general, solutions of the WDVV equation correspond to \ll  Frobenius manifolds\rrr, a generalization of quantum cohomology (see section 8.4 of \cite{CoKa99} and the references there).

Let us see how this leads to an integrable p.d.e.\ which admits the big quantum cohomology of $\C P^2$ as a distinguished solution.  Then we shall return to the matter of reconstructing the big quantum cohomology from the small quantum cohomology. This famous example is taken from
\cite{KoMa94}. 

We consider the product operation defined in the above way by a function $\F$ on the three-dimensional complex vector space  $H^\ast \C P^2 = \C 1 \oplus \C b \oplus \C b^2$.  
We assume that $1$ is the identity element; commutativity is automatic.
It follows that $b\ast_t 1=b=1\ast_t b$,  
$b\ast_t b=b^2 + \F_{111} b + \F_{112}1$, and
$b\ast_t b^2=b^2\ast_t b=\F_{121} b + \F_{222}1$,
$b^2\ast_t b^2=\F_{221} b + \F_{222}1$.  
There is just one nontrivial associativity condition in this example,  namely 
$( b\ast_t  b)\ast_t  b^2 =
 b\ast_t ( b\ast_t  b^2)$.
In terms of $\F$ this condition is  $\F_{222}+\F_{111}\F_{122}=\F_{112}^2$, which is by definition the WDVV equation.

Now, it turns out that the associativity condition is equivalent to the flatness of the connection $d+\tfrac1\h\om$  (this connection is defined in the same way as for small quantum cohomology).  
From the above products, we see that the connection form is given explicitly by
\[
\om=
\bp
 1  &  0  &  0 \\
0 & 1& 0\\
0 & 0 & 1
\ep
dt_0
+
\bp
  0  &  \F_{112}  &  \F_{122}\\
1 & \F_{111}& \F_{121}\\
0 & 1 & 0
\ep
dt_1
+
\bp
  0  &  \F_{212}  &  \F_{222}\\
0 & \F_{211}& \F_{221}\\
1 & 0 & 0
\ep
dt_2.
\]
This exhibits the WDVV equation as an integrable p.d.e.\ with spectral parameter $\h$ (cf.\ the end of section \ref{extensions}) for the function $\F$.  The particular solution given by the Gromov-Witten potential of $\C P^2$ turns out to be
\[
\F^{\C P^2}(t_0,t_1,t_2)= 
\tfrac12(t_0t_1^2 + t_0^2t_2) + \sum_{d\ge 1} N_d\, e^{dt_1} 
\tfrac{\vphantom{{t_2}_{A_A}}t_2^{3d-1}}{(3d-1)!}
\]
where the $N_d$
are determined recursively by $N_1=1$ and
\[
N_d=\sum_{i+j=d}
\left(
\tbinom{3d-4}{3i-2}i^2j^2 - i^3j\tbinom{3d-4}{3i-1}
\right)
N_iN_j.
\]
The positive integer $N_d$ can be interpreted as
the number of rational curves
of degree $d$ in $\C P^2$ which hit $3d-1$ generic points.
As a function of $t_0, q_1=e^{t_1}, t_2$ the series for $\F^{\C P^2}$
converges in a neighbourhood of the point $(t_0,q_1,t_2)=(0,0,0)$ (see
section 2 of \cite{DiIt95}).  

The Reconstruction Theorem of \cite{KoMa94} says that all of this highly nontrivial information may be \ll reconstructed\rr from the (much simpler) small quantum cohomology of $\C P^2$.  More precisely, any flat connection of the form
\[
\om=
\bp
 1 & 0 & 0\\
0 & 1& 0\\
0 & 0 & 1
\ep
dt_0
+
\bp
0  &  \ast  &  \ast\\
1 & \ast& \ast\\
0 & 1 & 0
\ep
dt_1
+
\bp
  0  &  \ast  &  \ast\\
0 & \ast& \ast\\
1 & 0 & 0
\ep
dt_2
\]
which satisfies the
\ll initial condition\rr 
\[
\om\vert_{t_0=t_2=0}\  =\ 
\bp
  1  &  0  &  0 \\
0 & 1& 0\\
0 & 0 & 1
\ep
dt_0
+
\bp
  0  &  0  &  e^{t_1}\\
1 & 0& 0\\
0 & 1 & 0
\ep
dt_1
+
\bp
  0  &  e^{t_1} &  0\\
0 & 0& e^{t_1} \\
1 & 0 & 0
\ep
dt_2
\]
must be of the previous form for some $\calF$, and, furthermore, $\calF$ is essentially unique.  
An elementary discussion of this can be found in \cite{DiIt95}.  More sophisticated and more general versions of this argument have been given, starting with \cite{HeMa04}. 

In terms of our extension procedure, this example can be formulated as follows.   If the D-module basis giving rise to the connection $d+\tfrac1{\h}\om$ is $[P_0]=[1],[P_1],[P_2]$ then the component $\tfrac1{\h}\om_1dt_1$ shows that
$P_1\equiv \h\b_1$ and $P_2\equiv (\h\b_1)^2 -\calF_{111}\h\b_1-\calF_{112}$.  Computing $\h\b_1[P_2]$ gives a third order relation
\[
T_1=
(\h\b_1)^3 - \calF_{111} (\h\b_1)^2 - 
(2\calF_{112}+\h\calF_{1111})  \h\b_1  - 
(\calF_{122}+\h \calF_{1112}).
\]
Similarly, the component $\tfrac1{\h}\om_2dt_2$ gives 
$\h\b_2[P_0]=[P_2]=[(\h\b_1)^2 -\calF_{111}\h\b_1-\calF_{112}]$, hence 
\[
T_2=\h\b_2 - P \text{\ \ where\ \ }
P=
(\h\b_1)^2 -\calF_{111}\h\b_1-\calF_{112}
\]
is also a relation.   These two relations generate the ideal of relations of the D-module.  This is an example of the situation of Proposition \ref{rank}.

\subsection{Crepant resolutions}

In \cite{CITXX},  two examples were given to illustrate a general principle known as the \ll Crepant Resolution Conjecture\rrr.  The simpler of the two relates the quantum cohomology of the Hirzebruch surface 
$\FF_2=\P( \calO(0)\oplus\calO(-2) )$ (a $\C P^1$-bundle over $\C P^1$; the fibrewise one point compactification of $T\C P^1$) to the quantum cohomology of the weighted projective space $\P(1,1,2)$ (the one point compactification of $T\C P^1$).  The natural map $\FF_2\to \P(1,1,2)$ is biholomorphic away from the singular point $[0,0,1]$ of $\P(1,1,2)$.  It is a crepant resolution, and the Crepant Resolution Conjecture predicts that the (orbifold) quantum cohomology of $\P(1,1,2)$ can be obtained by specializing the quantum parameters $q_1,q_2$ of $\FF_2$ to certain values. Coates et al confirm the conjecture in this case by comparing the D-modules of each space, and carefully matching up their $J$-functions after analytic continuation.    
Such examples are valuable as a guide to finding the most appropriate formulation of the  Crepant Resolution Conjecture, and more generally to understanding the functorial properties of quantum cohomology under birational maps.

We shall explain this example very simply using the method of section \ref{extensions}.  As this does not involve direct geometric arguments,  it suggests the possibility of a purely D-module theoretic formulation of the conjecture.  

First, we state the quantum D-modules of each space, which are well known.  Since $\FF_2$ is a $\C P^1$-bundle over $\C P^1$, $H^2 \FF_2$ has two additive generators, which we call $b_1,b_2$.  Geometrically their Poincar\'e duals may be represented by a fibre and the infinity section of the bundle, respectively. With respect to this basis, it can be shown (section 5 of \cite{Gu05} or chapter 11 of \cite{CoKa99}) that 
\[
\calM^{\FF_2}=
{D^\h_{t_1,t_2}}/{(F_1,F_2)}
\]
where
$F_1=(\h\b_1)^2-q_1q_2$,
$F_2=\h\b_2(\h\b_2-2\h\b_1) - q_2(1-q_1)$.  This is a \ll quantization\rr of $QH^\ast\FF_2=\C[b_1,b_2,q_1,q_2]/(b_1^2-q_1q_2,b_2(b_2-2b_1)- q_2(1-q_1))$.

For $\P(1,1,2)$, the (orbifold) quantum cohomology D-module was calculated in \cite{CCLTXX}.  The (orbifold) cohomology group  $H^2_{\orbi}\P(1,1,2)$ contains an obvious \ll hyperplane class\rr $b$.  With respect to this, one has
\[
\calM^{\P(1,1,2)}=
D^\h_t/(P)
\]
where $P=(\h\b)^4 - \tfrac12\h(\h\b)^3 - \tfrac14 q$.  

Now, $H^2_{\orbi}\P(1,1,2)$ has rank two; it has another additive generator called $\one$, which arises from the orbifold structure at the singular point $[0,0,1]$.  The definitions of 
orbifold cohomology and orbifold quantum cohomology are substantial generalizations of the non-orbifold case, and we shall not discuss them here.  However, the available evidence suggests that the orbifold quantum differential equations behave in a similar way to those in the non-orbifold case. In particular, the orbifold Gromov-Witten invariants of weighted projective space may be extracted by the method of section \ref{qde} --- see
\cite{GuSaXX}.   The canonical\footnote{\ll Canonical basis\rr means a basis constructed from a monomial basis by the canonical procedure of Theorem \ref{basis}.}
bases of the quantum D-modules  $\calM^{\FF_2}$, $\calM^{\P(1,1,2)}$, with their corresponding cohomology bases, are as follows:
\[
\begin{array}{c|lcc|l}
1 & 1 & \hphantom{aaaa} & 1 & 1
\\
b_1 & \h\b_1 & & b & \h\b
\\
b_2 & \h\b_2 & & \ b^2 & (\h\b)^2
\\
b_1b_2 & \h\b_1\h\b_2-q_1q_2 & & \ \one & 2q^{-1/2}(\h\b)^3
\end{array}
\]

\begin{theorem}{\cite{CITXX}} The orbifold quantum D-module
$\calM^{\P(1,1,2)}$  is obtained from the quantum D-module $\calM^{\FF_2}$
by setting $(q_1,q_2)=(-1,\i q^{1/2})$. This is a natural identification in which the basis $1,b_1,b_2, b_1b_2$ of $H^2\FF_2$ corresponds to the basis $1,b-\i\one,2b,2b^2$ of $H^2\P(1,1,2)$.
\end{theorem}

We can derive this very easily (with hindsight) by expressing the D-module in the form given at the end of section \ref{extensions}.  We begin by computing expressions for the powers of $\h\b_2$ by differentiating the relations $F_1,F_2$:
\begin{align*}
(a)\  (\h\b_2)^2&\equiv 2\h\b_1\h\b_2 + q_2(1-q_1)
\\
(b)\  (\h\b_2)^3&\equiv (3q_1q_2+q_2)\h\b_2+2q_2(1-q_1)\h\b_1+\h q_2(1+q_1)
\\
(c)\  (\h\b_2)^4&\equiv 2q_2(1+q_1)(\h\b_2)^2 + \h(\h\b_2)^3+
\h q_2(1+q_1)\h\b_2 - q_2^2(1-q_1)^2.
\end{align*}
From (a) and (b) we see that $[1],[\h\b_2],[(\h\b_2)^2],[(\h\b_2)^3]$ are linearly independent; they form a basis of $\calM^{\FF_2}$.  The fourth order relation
\[
T_1= (\h\b_2)^4 - 2q_2(1+q_1)(\h\b_2)^2 - \h(\h\b_2)^3 -
\h q_2(1+q_1)\h\b_2 + q_2^2(1-q_1)^2
\]
is given by (c).   From (b) we obtain
\[
\h\b_1\equiv \frac{1}{2q_2(1-q_1)}
\left(
(\h\b_2)^3 - (3q_1q_2+q_2)\h\b_2 - \h q_2(1+q_1)
\right) = P \ \text{(say)}.
\]
This gives a relation $T_2=\h\b_1-P$.  For dimensional reasons we must have 
\[
\calM^{\FF_2}=
{D^\h_{t_1,t_2}}/{(F_1,F_2)}=
{D^\h_{t_1,t_2}}/{(T_1,T_2)}.
\]
Let us now put $(q_1,q_2)=(-1,\i q^{1/2})$.  From $q_2=\i q^{1/2}$, we see that the operator 
$\b_2$ restricts to $2\b$, so the operator $T_1$ restricts to
\[
16\left(
(\h\b)^4 - \tfrac12\h (\h\b)^3 - \tfrac14 q
\right),
\]
which is the quantum differential operator of $\calM^{\P(1,1,2)}$, up to a scalar multiple.
Thus, we have exhibited $\calM^{\FF_2}$ as a $t_1$-extension of $\calM^{\P(1,1,2)}$, in the manner of section \ref{extensions}. 

It remains to extract the relation between the canonical bases from this description.  Under the specialization of variables, we have already seen that $\h\b_2$ restricts to $2\h\b$.  Next, formula (b) shows that the operator $\h\b_1$ restricts to 
\[
\frac{1}{4\i q^{1/2}}\left(
(2\h\b)^3 + 4\i q^{1/2} \h\b
\right)=
-2\i q^{-1/2}(\h\b)^3 + \h\b.
\]
Finally, since $(\h\b_2)^2$ restricts to $4(\h\b)^2$, we see from relation $F_2$ that $\h\b_1\h\b_2$ restricts
to $2(\h\b)^2 -\i  q^{1/2}$, hence $\h\b_1\h\b_2-q_1q_2$ restricts to $2(\h\b)^2$. 
From the table above, we read off that the basis elements $1,b_1,b_2, b_1b_2$ correspond to $1,b-\i\one,2b,2b^2$,
as required.   

The correspondence between the bases may be justified geometrically, by examining the map $\FF_2\to \P(1,1,2)$, but it is remarkable that the quantum D-module contains this information implicitly --- along with all the Gromov-Witten invariants of both spaces.  For a recent update on the conjecture we refer to \cite{IrYY}.

\subsection{Harmonic maps and mirror symmetry}

The second candidate for an integrable system whose solutions include quantum cohomology is the harmonic (or pluriharmonic) map equation. Small quantum cohomology is sufficient for this, but, as in the case of the WDVV equation, an entirely new direction --- this time towards mirror symmetry --- is required.
Further details of the discussion below may be found in chapter 10 of \cite{Gu08}.

\no{\em The harmonic map equation.}

The equation for a harmonic map $\phi:\R^2=\C\to G$, where $G$ is a (compact or noncompact) Lie group, is
 \[
\bx(\phi^{-1}\bx \phi) + \by(\phi^{-1}\by \phi)=0.
\]
Writing $z=x+i y$ and
$\b = {\b}/{\b z}=\tfrac12({\b}/{\b x} - i {\b}/{\b y}),$
$\db = {\b}/{\b \bar{z}}=\tfrac12({\b}/{\b x} +i {\b}/{\b y}),$
the  equation becomes
\[
\b(\phi^{-1}\db \phi) + \db(\phi^{-1}\b \phi)=0.
\]
This notation assumes that $G$ is a matrix group.  If
$\GC$ is the complexification of $G$, and $C:\GC\to\GC$ is the natural 
conjugation\footnote{If $G=\U_n$, then $\GC=\glnc$, and 
$C:\GC\to\GC$, $c:\gc\to\gc$ are given respectively by
$C(A)={A^\ast}^{-1}$, $c(A)=-A^\ast$.}  map, and $c:\gc\to\gc$ is the induced conjugation map of Lie algebras, then  $\phi^{-1}\db \phi,\phi^{-1}\b \phi$ take values in $\gc$ and satisfy $c(\phi^{-1}\b \phi)=\phi^{-1}\db \phi$.

The harmonic map equation can be represented as an integrable p.d.e.\  with spectral parameter if we introduce 
the $\gc$-valued $1$-form
\[
\al=
\tfrac12(1-\tfrac{1}{\la}) (\phi^{-1}\b_1 \phi) dz_1 +
\tfrac12(1-{\la}) (\phi^{-1}\b_2 \phi) dz_2
\]
where $\la$ is a complex parameter. Namely, the connection $d+\al$ is flat 
for every (nonzero) value of $\la$ if and only if $\phi$ satisfies the harmonic map equation.  In fact, it is well known and easy to prove (see sections 4.3 and 7.3 of \cite{Gu08}) that the following more general statement holds:

\begin{proposition} Let 
$
\al=\tfrac12(1-\tfrac{1}{\la})\al_1dz_1+
\tfrac12(1-{\la})\al_2dz_2
$
be a $\gc$-valued $1$-form on $\C^2$ (or a simply connected open subset of $\C^2$).  If $d+\al$ is flat for every (nonzero) value of $\la$, then there exists a map $\phi:\C^2\to \GC$ such that
$\al_1=\phi^{-1}\b_1 \phi,\quad
\al_2=\phi^{-1}\b_2 \phi$
and this map satisfies the equation
\[
\b_1(\phi^{-1}\b_2 \phi) + \b_2(\phi^{-1}\b_1 \phi)=0.
\]
Conversely, let $\phi:\C^2\to \GC$ be a map which satisfies the equation
$\b_1(\phi^{-1}\b_2 \phi) + \b_2(\phi^{-1}\b_1 \phi)=0$. Then the $1$-form $\al=
\tfrac12(1-\tfrac{1}{\la}) (\phi^{-1}\b_1 \phi) dz_1 +
\tfrac12(1-{\la}) (\phi^{-1}\b_2 \phi) dz_2$
defines a flat connection $d+\al$.

This remains true when the reality conditions
$z_1=z, z_2=\bar z, \ \ c(\al_1)=\al_2$ are imposed,
giving the harmonic map equation (on $\C$, or a simply connected open 
subset of $\C$).
\end{proposition}

The spectral parameter here plays a crucial role, because $\al$ may be regarded as a $1$-form taking values in the based loop algebra $\Om\g$, and the flatness condition implies (Theorem \ref{flat}) that $\al=\Phi^{-1}d\Phi$ for some $\Phi:N\to \Om G$ (on a simply connected open subset $N$ of $\C$).   This $\Phi$ is called an extended solution, or extended harmonic map.  Moreover, the shape of $\al$ implies that $\Phi$ is holomorphic with respect to the natural complex structure of the based loop group $\Om G$. Since this complex structure may be obtained from an identification of $\Om G$ with an open subset\footnote{See section 8.8 of \cite{Gu08}.
If $G$ is compact, $\Om G$ may be identified with $\La\GC/\La_+\GC$.}
of $\La\GC/\La_+\GC$. 
it follows that $\Phi=[L]$ for some holomorphic map $L:N\to \La\GC$, i.e.\ $L=\Phi B$ for some (smooth) $B:N\to\La_+\GC$.  This $L$ is of course not unique, but there is a canonical choice for it, obtained from the Birkhoff factorization
$\Phi=\Phi_-\Phi_+$ and taking $L=\Phi_-, B=\Phi_+^{-1}$. It can then be shown 
(section 7.3 of \cite{Gu08}) that $L$ satisfies an equation of the form
\[
L^{-1}dL =\tfrac1\la\om
\]
where $\om$ is a holomorphic $\gc$-valued $1$-form on $N$. Conversely, if $\om$ is {\em any} holomorphic $\gc$-valued $1$-form, then  $\tfrac1\la \om$ is of the form $L^{-1}dL$ (since $\om$ depends only on a single variable $z$, and is therefore flat).  From the Iwasawa factorization $L=L_\R L_+$, it is easy to show that the map given by $\Phi=L_\R$ is an extended harmonic map.  This correspondence 
\[
\phi \longleftrightarrow \om
\]
between harmonic maps $\phi$ and \ll unrestricted holomorphic data\rr $\om$ is known as the DPW correspondence, or generalized Weierstrass representation.  Further details can be found in section 7.3 of \cite{Gu08} or the original paper \cite{DoPeWu98}.

\no{\em The D-module.}

The harmonic map equation in the form $L^{-1}dL =\tfrac1\la\om$ can be rewritten in an illuminating way if we make use of the Grassmannian model of $\Om G$. This is an identification
\[
\Om G \ \cong\  \Gr^\g
\]
of $\Om G$ with a certain infinite-dimensional Grassmannian manifold (section 8.6 of \cite{PrSe86}).  The holomorphic map $\Phi=[L]:N\to \Om G$ corresponds to a 
holomorphic map $W:N\to\Gr^\g$, and hence to a holomorphic vector bundle 
$W^\ast\calT$ where $\calT$ is the tautologous vector bundle on $\Gr^\g$. The 
harmonic map equation $L^{-1}dL =\tfrac1\la\om$ can then be written as
\[
\la \b_z\ \Ga W^\ast\calT  \sub  \Ga W^\ast\calT 
\]
where $\Ga$ denotes the space of holomorphic sections.  This says that 
$\Ga W^\ast\calT$ has a D-module structure: it is acted upon by the ring 
$D^\la_z$ of differential operators which is generated by $\la \b_z$ and whose elements have coefficients which are holomorphic in $z\in N$ and holomorphic in $\la$ in a neighbourhood of $\la=0$.   In the case $G=\U_n$ this is explained in detail in section 8.2 of \cite{Gu08}.  

This D-module does not generally have a distinguished cyclic generator, although some examples with distinguished cyclic generators can be found in \cite{Gu02} (Proposition 2.3 and Theorem 2.4).  We shall focus on one particular kind of harmonic map which arises from quantum cohomology, where the D-module may be identified with the quantum D-module.

\no{\em Quantum cohomology as a (pluri)harmonic map.}

To explain the link with quantum cohomology, two extensions are needed (sections 7.4,7.5 of \cite{Gu08}).  First, the theory applies also to pluriharmonic maps $\phi:\C^r\to G$, whose equations have a similar zero curvature form.  However, when $r>1$, the holomorphic data
$\om= \sum_{i=1}^r\om_i dz_i$ is no longer \ll unrestricted\rrr; it is subject to the nontrivial  flatness condition $d\om=\om\wedge\om=0$.  Second,  the theory applies to harmonic (or pluriharmonic) maps $\phi:\C\to G/K$ where $G/K$ is a symmetric space.  Here the $1$-form $\al$ looks simpler as it can be written
\[
\al= (\al_1^\k + \tfrac{1}{\la}\al_1^\m) dz+
(\al_2^\k + {\la}\al_2^\m) d\bar z,
\]
where $\al_i=\al_i^\k + \al_i^\m$ denotes the eigenspace decomposition of the involution $\si:\gc\to\gc$ which defines the symmetric space.

The first main observation is that the map $L:N\to\glsc$ of quantum cohomology (where $N$ is an open subset of $H^2 M\cong \C^r$) is exactly of the above form, that is, it satisfies 
$L^{-1}dL =\tfrac1\la\om$ where $\om$ is given by the quantum products.   Moreover, $\om$ takes values in $\mc$, the $(-1)$-eigenspace of a certain natural involution $\si$ on $\gc=\frakglsc$ (this fact corresponds to the Frobenius property of the quantum product). 
By the above general theory, it follows that the quantum cohomology of $M$ defines a pluriharmonic map  into a symmetric space $G/K$, where $G$ is any real form of $\glsc$. 
One natural real form\footnote{In chapter 10 of \cite{Gu08}, an (indefinite) unitary group based on $H^\ast(M;\R)$
was used.  Any choice gives a pluriharmonic map, so the \ll best\rr choice depends on imposing further criteria.  
Further discussion of this point, in particular the relation with  \cite{IrXX}, can be found in section 6 of \cite{DGRXX}.}
 is $\glsr$, corresponding to the cohomology with real coefficients $H^\ast(M;\R)$, and this gives the symmetric space $\glsc/\O_{s+1}$.

The second main observation concerning quantum cohomology --- and the link with mirror symmetry --- is that in certain situations this (pluri)harmonic map has an independent geometrical interpretation, as the period map for a variation of Hodge structure.  The most famous example is the quintic threefold $M$ in $\C P^4$.  The harmonic map obtained from the quantum cohomology of $M$ can be described very simply as follows:  for a certain holomorphic $\C^4$-valued function $u$, consider the holomorphic map
\[
U\quad = \quad
\Span \{u\} \sub \Span \{u, u^\pr\} \sub 
\Span \{u,u^\pr, u^{\pr\pr}\}
\sub \C^4
\]
to the flag manifold
$\SU_4/\S(\U_1\times \U_1\times \U_1\times \U_1)$.  This flag manifold can be identified with the space of quadruples $(L_1,L_2,L_3,L_4)$ of mutually orthogonal complex lines in the Hermitian space $\C^4$, and it is well known (cf.\ Example 8.16 of \cite{Gu08}) that the composition of any map $U$ of the above form with the projection map $(L_1,L_2,L_3,L_4)\mapsto L_1\oplus L_3$ is a harmonic map into the symmetric space 
$\Gr_2\C^4=
\SU_4/\S(\U_2\times \U_2)$.  In terms of the general theory of section \ref{qde}, the map $u$ can be identified with the $J$-function, and the map $U$ with $L$,  if the flag manifold is embedded suitably in the loop group $\Om \SU_4$. (This is consistent with the above choice of symmetric space $\GL_{4}\R/\O_4$, because $U$ actually takes values in the smaller symmetric space $\Sp_4\R/\U_2$, a symplectic Grassmannian.)   Mirror symmetry says that $U$ can be identified with the period map
\[
H^{3,0}\sub   
H^{3,0}\!\oplus\! H^{2,1}\sub 
H^{3,0}\!\oplus\! H^{2,1}\!\oplus\! H^{1,2}\sub 
H^{3,0}\!\oplus\! H^{2,1}\!\oplus\! H^{1,2}\!\oplus\! H^{0,3}
\]
where $H^{i,j}=H^{i,j} \tilde M$ for a \ll mirror partner\rr $\tilde M$ of $M$.  The domain of this map is (an open subset of) the moduli space of complex structures of $\tilde M$, as $H^{i,j} \tilde M$ depends on the complex structure.  

It is a special feature of Calabi-Yau manifolds (such as the quintic threefold) that the harmonic map can be described as above in elementary terms, without using loop group theory. However, for the quantum cohomology of Fano manifolds (such as $\C P^n$), the map $L$ does not factor through a finite-dimensional submanifold of the loop group.  It does still have a variation of Hodge structure interpretation, in a generalized sense (due to Barannikov
\cite{Bar01}, \cite{Bar02a} and Katzarkov et al
\cite{KaKoPaXX}), because of the Grassmannian model of the loop group: instead of $U$, we use the holomorphic map $W$ (associated to $L$) which was described above.

The correspondence
\[
\phi \ \text{(variation of Hodge structure)}\ 
\longleftrightarrow 
\om \ \text{(quantum cohomology)}\ 
\]
between the pluriharmonic map $\phi$ and the holomorphic data $\om$
can be regarded as an expression of mirror symmetry. It is given explicitly by the Birkhoff and Iwasawa loop factorizations:
\[
\begin{matrix}
 & \xrightarrow{ \text{Birkhoff} } & \\
 \Phi\ (\text{or}\ \phi) \ \ \ & & \ \ \ L\ (\text{or}\ \om) \\ 
  & \xleftarrow{ \text{Iwasawa} } &
\end{matrix}
\]
In this way, the quantum D-module contains not only the geometric information consisting of the Gromov-Witten invariants, but also the much less visible geometric information consisting of the variation of Hodge structure of the mirror partner.

\section{Conclusion}\label{conclusion}

Much remains to be done to clarify the integrable systems aspects of quantum cohomology, but an even more elusive goal is to {\em characterize} quantum cohomology in purely differential equation theoretic terms.  The quantum D-module will attain the status of de Rham cohomology (for example) only if those D-modules which occur as quantum D-modules can be described precisely.
This goal is probably too optimistic, but one can at least make a start by listing some conditions, such as:

---quantization of a commutative algebra

---regular singular point of maximal unipotent monodromy at $q=0$

---homogeneity

---self-adjointness

\no In section \ref{qde}
we have focused on the first of these, so let us comment briefly on the others, taking the case of $M^3_5$ 
(Examples \ref{hyper}, \ref{hyper2}, \ref{hyper3})
as a concrete example.  The quantum differential operator here is
\[
(\h\b)^4-27q(\h\b)^2-27\h q(\h\b)-6\h^2q.
\]
The second condition has the usual meaning from o.d.e.\ theory. The third means that the operator is weighted homogeneous, the weights of the symbols $\b, \h, q$ being
$0, 2, 4$ respectively.  So far,  any quantization of the quantum cohomology relation $b^4-27qb^2$ of the form 
$(\h\b)^4-27q(\h\b)^2-\al\h q(\h\b)-\be\h^2q$
would have all these properties, where $\al,\be$ are constants. The fourth condition means that the operator is formally self-adjoint with respect to the involution defined by $\b^\ast=-\b$, $\h^\ast=-\h$ (see section 6.3 of \cite{Gu08}).  This condition forces $\al$ to be $27$.  However, only the value $\be=6$ gives the correct quantum products (or Gromov-Witten invariants), and our conditions do not pin this down; we need more, and these are not going to be straightforward. 

One source of additional conditions is the global behaviour of the associated (pluri)harmonic map, regarded as a generalized period map.  There is a positivity condition which generalizes the second Riemann-Hodge bilinear relation, and it is natural to insist on this.  Some ideas in this direction can be found in \cite{IrXX} (see also \cite{DGRXX} for a particular example). A related source is the arithmetic behaviour of the differential equation. Even in the Calabi-Yau case where the second Riemann-Hodge bilinear relation is known to hold, it is very difficult to characterize differential equations whose solutions have the expected integrality properties (related to the integrality or rationality of the Gromov-Witten invariants) --- see, for example,
\cite{AmvEvSZuXX} and \cite{AmZuXX}.

{\em
\no
Department of Mathematics and Information Sciences\newline
   Faculty of Science and Engineering\newline
   Tokyo Metropolitan University\newline
   Minami-Ohsawa 1-1, Hachioji, Tokyo 192-0397\newline
   JAPAN}


\begin{thebibliography}{99}

\bibitem{AmvEvSZuXX}
G.~Almkvist, C.~van~Enckevort, D.~van~Straten,  and W.~Zudilin,
\emph{Tables of Calabi--Yau equations},
preprint,
math.AG/0507430.

\bibitem{AmZuXX}
G.~Almkvist and W.~Zudilin,
\emph{Differential equations, mirror maps and zeta values},
Mirror symmetry V,
AMS/IP Stud. Adv. Math. 38,
Amer. Math. Soc.,
2006,
pp.~481--515.
(math.NT/0402386)

\bibitem{Bar01}
S.~Barannikov,
\emph{Quantum periods. I. Semi-infinite variations of Hodge structures},
Internat. Math. Res. Notices
\textbf{2001-23}
(2001),
1243--1264.

\bibitem{Bar02a}
S.~Barannikov,
\emph{Non-commutative periods and mirror symmetry in higher dimensions},
Comm. Math. Phys.
\textbf{228}
(2002),
281--325.

\bibitem{Bj79}
J.~Bj\"ork,
\emph{Rings of Differential Operators},
North-Holland,
1979.

\bibitem{CITXX}
T.~Coates, H.~Iritani, and H.-H.~Tseng,
\emph{Wall-crossings in toric Gromov-Witten theory I: crepant examples},
preprint,
math.AG/0611550.

\bibitem{CCLTXX}
T.~Coates, A.~Corti, Y.-P.~Lee, and H.-H.~Tseng,
\emph{The quantum orbifold cohomology of weighted projective space},
preprint,
math.AG/0608481.

\bibitem{CoJi99}
A.~Collino and M.~Jinzenji,
\emph{On the structure of the small quantum cohomology rings of projective hypersurfaces},
Comm. Math. Phys.
\textbf{206}
(1999),
157--183.

\bibitem{CoKa99}
D.~A.~Cox and S.~Katz,
\emph{Mirror Symmetry and Algebraic Geometry},
Math. Surveys and Monographs 68,
Amer. Math. Soc.,
1999.

\bibitem{DiIt95}
P.~Di~Francesco and C.~ Itzykson
\emph{Quantum intersection rings},
The Moduli Space of Curves,
Prog. Math. 129,
eds. R.~H.~Dijkgraaf et al,
Birkh\"auser,
1995,
pp.~81--148.

\bibitem{DGRXX}
J.~Dorfmeister, M.~Guest, and W.~Rossman,
\emph{The $tt^\ast$ structure of the quantum cohomology of $\C P^1$ from the viewpoint of differential geometry},
preprint, arXiv:0905.3876. 

\bibitem{DoPeWu98}
J.~Dorfmeister, F.~Pedit, and H.~Wu,
\emph{Weierstrass type representations of harmonic maps into symmetric spaces},
Comm. Anal. Geom.
\textbf{6}
(1998),
633--668.

\bibitem{Gi95-1}
A.~B.~Givental,
\emph{Homological geometry and mirror symmetry},
Proc. Int. Congress of Math. I, Z\"urich 1994,
ed. S.~D.~Chatterji,
Birkh\"auser,
1995,
pp.~472--480.

\bibitem{Gi95-2}
A.~B.~Givental,
\emph{Homological geometry I.  Projective hypersurfaces},
Selecta Math.
\textbf{1}
(1995),
325--345.

\bibitem  {Gi96}
A.~B.~Givental,
\emph{Equivariant Gromov-Witten invariants},
Internat. Math. Res. Notices
\textbf{1996-13}
(1996),
1--63.
(alg-geom/9603021)

\bibitem{Gu02}
M.~A.~Guest,
\emph{An update on harmonic maps of finite uniton number, via the zero curvature equation},
Contemp. Math. 
\textbf{309}
(2002),
85--113.

\bibitem{Gu05}
M.~A.~Guest,
\emph{Quantum cohomology via D-modules},
Topology 
\textbf{44}
(2005),
263--281.
(math.DG/0206212)

\bibitem{Gu06}
M.~A.~Guest,
\emph{Introduction to homological geometry: I,II},
Integrable Systems, Geometry, and Topology, 
ed. C.-L.~Terng,
AMS/IP Studies in Advanced Mathematics 36,
Amer. Math. Soc. and International Press,
2006,
pp.~83--121, 123-150.
(math.DG/0104274, math.DG/0105032)

\bibitem{Gu08}
M.~A.~Guest,
\emph{From Quantum Cohomology to Integrable Systems},
Oxford Univ. Press,
2008.

\bibitem{GuSaXX}
M.~A.~Guest and H.~Sakai,
\emph{Orbifold quantum D-modules associated to
weighted projective spaces},
preprint, arXiv:0810.4236.

\bibitem{HeMa04}
C.~Hertling and Y.~Manin,
\emph{Unfoldings of meromorphic connections and a construction of Frobenius manifolds},
Frobenius manifolds. Quantum cohomology and singularities,
Aspects of Math. E 36,
eds. C.~Hertling et al,
Vieweg,
2004,
pp.~113--144.
(math.AG/0207089)

\bibitem{IrXX}
H.~Iritani,
\emph{Real and integral structures in quantum cohomology I: toric orbifolds},
preprint,
arXiv:0712.2204.

\bibitem{IrYY}
H.~Iritani,
\emph{Ruan's conjecture and integral structures in quantum cohomology},
preprint,
arXiv:0809.2749.

\bibitem{Ji02}
M.~Jinzenji,
\emph{Gauss-Manin system and the virtual structure constants},
Internat. J. Math
\textbf{13}
(2002),
445--477.

\bibitem{KaKoPaXX}
L.~Katzarkov, M.~Kontsevich, and T.~Pantev,
\emph{Hodge theoretic aspects of mirror symmetry},
preprint,
arXiv:0806.0107.

\bibitem{KoMa94}
M.~Kontsevich and Y.~Manin,
\emph{Gromov-Witten classes, quantum cohomology, and enumerative geometry},
Commun. Math. Phys.
\textbf{164}
(1994),
525--562.
(hep-th/9402147)

\bibitem{Ph79}
F.~Pham,
\emph{Singularit\'es des syst\`emes 
diff\'erentiels de Gauss-Manin},
Progr. Math. 2,
Birkh\"auser,
1979.

\bibitem{PrSe86}
A.~N.~Pressley and G.~B.~Segal,
\emph{Loop Groups},
Oxford Univ. Press,
1986.

\bibitem{PuSi03}
M.~van~der~Put, and M.~F.~Singer,
\emph{Galois Theory of Linear Differential Equations},
Grundlehren der Mathematischen Wissenschaften 328,
Springer,
2003.

\bibitem{SiTi97}
B.~Siebert and G.~Tian,
\emph{On quantum cohomology rings of Fano manifolds and
a formula of Vafa and Intriligator},
Asian J. Math.
\textbf{1}
(1997),
679--695.
(alg-geom/9403010)


\end{thebibliography}
\end{document}